\documentclass{article}
\usepackage{latexsym,amsfonts,euscript,amsmath,graphicx}
\usepackage{amssymb}
\usepackage{amsthm}
\usepackage{stmaryrd}
\usepackage{mathrsfs}
\usepackage{leftidx}
\usepackage{indentfirst}
\usepackage{color}

\usepackage[colorlinks,
linkcolor=blue,
anchorcolor=blue,
citecolor=blue]{hyperref} 

\makeatletter 
\@addtoreset{equation}{section}
\makeatother  

\def\la{\langle}\def\ra{\rangle}

\newcommand{\Aut}{{\operatorname{Aut}}}
\newcommand{\Out}{{\operatorname{Out}}}

\newcommand{\GL}{{\operatorname{GL}}}
\newcommand{\SL}{{\operatorname{SL}}}

\def\irr#1{{\rm Irr}(#1)}
\def\cent#1#2{{\bf C}_{#1}(#2)}
\def\syl#1#2{{\rm Syl}_#1(#2)}

\def\fitt#1{{\bf F}(#1)}
\def\z#1{{\bf Z}(#1)}

\def\o#1{\overline{#1}}
\def\gal#1#2{\operatorname{Gal}(#1/#2)}
\def\irrn#1{\operatorname{Irr}_{\mathfrak{n}}(#1)}
\def\irrs#1{\operatorname{Irr}_{\mathfrak{s}}(#1)}
\def\inert#1#2{\operatorname{I}_{#1}(#2)}

\newtheorem{lem}{ \bf Lemma}[section]

\newtheorem{pro}[lem]{\bf Proposition}
\newtheorem{thm}[lem]{\bf Theorem}
\newtheorem*{thm*}{\bf Theorem}
\newtheorem*{con*}{\bf Conjecture}
\newtheorem*{thmA}{\bf Theorem~A}
\newtheorem*{corB}{\bf Corollary~B}

\newtheorem{hy}[lem]{\bf Hypothesis}

\title{A characterization of finite groups by certain Galois conjugacy class of irreducible characters
\thanks{{\bf Acknowledgement:} The authors are supported by Chinese Scholarship Council, and the second author is supported by the NSF of China (No. 11971391, 12071376).
The authors also would like to thank Prof. Dolfi and Prof. Qian for many useful conversations on this topic.}}

 \author{Yu Zeng\footnote{Email: yuzeng2004@163.com}\qquad Dongfang Yang\footnote{Corresponding author. Email: dfyang1228@163.com}\\
  {\footnotesize\small  Dept. Mathematics, Changshu Institute of
  Technology, Changshu, Jiangsu, 215500, China}}

\date{}

\begin{document}
\maketitle

\vskip 1cm

\begin{center}\textbf{Abstract}\end{center}
We classify the finite groups $G$ which satisfies the condition that every complex irreducible character,
whose degree's square doesn't divide the index of its kernel in $G$, lies in the same Galois conjugacy class.

\vskip 5cm

\bigskip


\textbf{2020 MR Subject Classification}\,\, 20C15.
\pagebreak

\section{Introduction}

For a finite group $G$, we define a partition of $\irr{G}$, the \emph{set of all complex irreducible characters} of $G$, to be $\{\irrn{G},\irrs{G}\}$, where
$$\irrn{G}=\{\chi\in\irr{G}:\chi(1)^2\mid |G:\ker\chi|\},~~~~\irrs{G}=\irr{G}-\irrn{G}.$$
To some extent, $|\irrs{G}|$, the size of $\irrs{G}$, measures how far a finite group $G$ to be a nilpotent group.
A classical result in character theory of finite groups states that if $G$ is a finite nilpotent group, then $|\irrs{G}|=0$.
Gagola and Lewis generalized this result by showing that a finite group $G$ is nilpotent if and only if $|\irrs{G}|=0$, see \cite{gagolalewis1999};
Lv, Yang and Qian classified finite groups $G$ with $|\irrs{G}|=1$ which turn out to be solvable groups with Fitting height $2$, see \cite{lvyangqian}.

If $G$ is a finite group with $n$ a multiple of $|G|$,
then $\mathfrak{G}=\gal{\mathbb{Q}_n}{\mathbb{Q}}$ (the \emph{Galois group of the $n$-th cyclotomic extension})
naturally acts on the set $\irr G$.
In the next theorem, we weaken the condition in \cite{lvyangqian} and classify finite groups $G$ in which $\irrs{G}$ forms a Galois conjugacy class.

\begin{thmA}
	Let $G$ be a finite group.
	\begin{description}
		\item[(a)] Assume that $\irrs{G}$ forms a Galois conjugacy class. 
		Then $G=P\rtimes H$ where $P$, the nilpotent residue of $G$, is a Sylow $p$-subgroup of $G$ for some prime $p$;
		$\ker(\chi)=\cent{H}{P}\times U$, where $U=\Phi(P)=P'$, for all $\chi\in \irrs{G}$;
		$H/\cent{H}{P}$ acts Frobeniusly and irreducibly on $P/U\cong (C_p)^n$
		such that $|H/\cent{H}{P}|=(p^{n}-1)/d$ where $d=|\irrs{G}|$ divides $p-1$;
		$\z{\GL(P/U)}(H/\cent{H}{P})$ acts transitively on nontrivial elements of $P/U$ when identifying $H/\cent{H}{P}$ as a subgroup of $\GL(P/U)$;
		the field of values of $\chi$ is contained in the $p$-th cyclotomic field for all $\chi\in \irrs{G}$;
		and we are in one of the following cases.
		\let\thefootnote\relax\footnote{$(C_p)^n$ is the elementary abelian $p$-group of order $p^n$;}
		\let\thefootnote\relax\footnote{$E(p^{1+2})$ is either the quaternion group of order $8$ when $p=2$ or the Heisenberg group of order $p^3$ when $p > 2$.}
		\begin{description}
			\item[(a1)] $P\cong (C_p)^n$, $\cent{H}{P}=1$ and $H\cong C_{(p^n-1)/d}$.
			\item[(a2)] $P\cong (C_p)^2$ where $p$ is a Mersenne prime, $\cent{H}{P}=1$ and $H=Q\times D$ where $Q\in\syl{2}{H}$ is a generalized quaternion $2$-group and $D$ is cyclic normal $2$-complement of $Q$ in $H$.
			\item[(a3)] $P\cong (C_p)^n$, $\cent{H}{P}>1$, $H\in\syl{q}{G}$ and $H/\cent{H}{P}$ has prime order $q=1+p+\cdots+p^{n-1}$.
			\item[(a4)] $P\cong E(p^{1+2})$, $\cent{H}{P}=1$, and $H\cong C_{2(p+1)}$ where $|H:\cent{H}{U}|=2$.
			\item[(a5)] $P\cong E(p^{1+2})$, $\cent{H}{P}=1$, and $H=\cent{H}{U}\cong C_{p+1}$.
			\item[(a6)] $P\cong E(p^{1+2})$ where $p$ a Mersenne prime, $\cent{H}{P}=1$, and $H=\cent{H}{U}$ is a generalized quaternion $2$-group of order $(p^2-1)/d$ where $d\in\{(p-1)/2,p-1\}$.
			\item[(a7)] $P\cong E(2^{1+2})$, $\cent{H}{P}>1$, and $H\in\syl{3}{G}$ and $G/\cent{H}{P}\cong \SL (2,3)$.
		\end{description}
		\item[(b)] The converse of the above result is true.
	\end{description}
\end{thmA}

The next corollary follows immediately from Theorem A.

\begin{corB}
	If $G$ is a finite group in which $\irrs{G}$ forms a Galois conjugacy class, then $G$ is solvable with Fitting height $2$.
\end{corB}

The paper is organized as follows: in Section 2 we collect some useful results and in Section 3 we prove our main theorem.

In the following, all groups considered are finite and $p$ always denotes a prime.
We use standard notation in character theory, as in \cite{isaacs1994}.

\section{Preliminaries}

We will use the following notations: $Q_{2^m}$, where $m\geq 3$, is the generalized quaternion $2$-group of order $2^m$; $(C_p)^n$ is the elementary abelian $p$-group of order $p^n$ for a positive integer $n$; and $E(p^{1+2})$ is either the quaternion group of order $8$ when $p=2$ or the Heisenberg group of order $p^3$ when $p>2$.
Also, we use $\mathrm{M}(G)$ to denote the Schur multiplier of a group $G$ and
$\mathbb{M}$ to denote the set of Mersenne primes.


\begin{lem}[\cite{zsig1892}]\label{zsig}
	Let $p$ be a prime and $n$ a positive integer. Then Zsigmondy prime divisor exists for $p^n-1$ except $p^n=2$; $n=2$ and $p\in\mathbb{M}$; or $p^n=2^6$.
\end{lem}


\begin{lem}[\cite{isaacs1999}]\label{isaacs}
	If a nontrivial nilpotent group $N$ acts faithfully and coprimely on a group $G$,
	then there exists some $g\in G$ such that $|\cent{N}{g}|\leq (|N|/p)^{1/p}$ where $p$ is the minimal prime factor of $|N|$.
\end{lem}


\begin{lem}[\cite{lvyangqian}]\label{Frob}
	Let a nontrivial nilpotent group $H$ act faithfully and irreducibly on an elementary abelian $p$-group $V$ and let $G=V\rtimes H$.
	Assume that either $\cent{H}{v}=1$ or $|H|~\big|~ |\cent{H}{v}|^2$ for each $ v\in V$.
	Then $G$ is a Frobenius group with complement $H$ and kernel $V$.
\end{lem}

\begin{lem}\label{specialextra}
	Let $P$ be a $p$-group such that $\Phi(P)=P'=\z{P}$.
	If $|P:P'|=p^2$, then $P$ is an extraspecial $p$-group of order $p^3$.
\end{lem}
\begin{proof}
	Since $P$ is a $p$-group such that $\Phi(P)=P'=\z{P}$ has index $p^2$ in $P$,
	$P$ possesses an abelian maximal subgroup $A$,
	and hence $|A|=|P'||\z P|=|P'|^2$.
	As $|A:P'|=p$, $|P'|=p$.
	Consequently, $P$ is an extraspecial $p$-group of order $p^3$.
\end{proof}

We say a group $H$ acts \emph{Frobeniusly} on a group $K$ if $H$ is nontrivial such that $\cent{K}{h}=1$ for every nontrivial element $h\in H$.
In the next three lemmas, we deal with faithful coprime action of a nilpotent group on a $p$-group in some special circumstances.

\begin{lem}\label{Frobcomp}
	Let a nilpotent group $H$ act Frobeniusly and irreducibly on an elementary abelian $p$-group $V$ and let $G=V\rtimes H$.
	Write $|V|=p^n$ for some integer $n$.
	Assume that $|H|=(p^n-1)/d$ where $d\mid p-1$.
	Then
	\begin{description}
		\item[(a)] Assume that Zsigmondy prime divisor of $p^n-1$ exists. Then $H\cong C_{(p^n-1)/d}$.
		\item[(b)] Assume that Zsigmondy prime divisor of $p^n-1$ doesn't exist. Then we are in one of the following cases.
		      \begin{description}
			      \item[(b1)] Either $H\cong C_{(p^2-1)/d}$ where $p\in \mathbb{M}$ or $H\cong C_{63}$ where $p^n=2^6$.
			      \item[(b2)] $p$ is a Mersenne prime and either $H\cong Q_{p+1}\times C_{(p-1)/d}$ where $2\mid d$ and $p\geq 7$ or $H\cong Q_{2(p+1)}\times C_{(p-1)/2d}$ where $2\nmid d$ and $p\geq 3$.
		      \end{description}
	\end{description}
\end{lem}
\begin{proof}
	Assume that $q$ is a Zsigmondy prime divisor of $p^n-1$.
	Let $Q$ be a subgroup of $\z {H}$ of order $q$.
	Then $H$ can be embedded in the group of units of $\mathrm{End}_{\mathbb{F}_p[Q]}(V)$.
	Note that $Q$ acts irreducibly on $V$, and so $H$ is cyclic by Schur's lemma and Wedderburn's little theorem.

	Assume that Zsigmondy prime divisor of $p^n-1$ doesn't exist.
	Since $G$ is a Frobenius group with kernel $V\cong (C_p)^n$,
	it follows by Lemma \ref{zsig} that either $n=2$ and $p\in\mathbb{M}$ or $p^n=2^6$.
	Note that $H$ is a nilpotent Frobenius complement, and hence $H$ is either cyclic or generalized quaternion by cyclic.
	Thus, by calculation either (b1) or (b2) holds.
\end{proof}

\begin{lem}\label{exspec}
	Let a nilpotent group $H$ act coprimely on a nonabelian $p$-group $P$ and let $G=P\rtimes H$.
	Assume that $U$ and $P/U\cong (C_p)^n$ are the $G$-chief factors of $P$ and that $H$ acts Frobeniusly on $P/U$.
	If $|H|=(p^n-1)/d$, where $d\mid p-1$, divides $|\cent{H}{u}|^2$ for each $u\in U$, then $P\cong E(p^{1+2})$ and one of the following holds.
	\begin{description}
		\item[(a)] Either $H=\cent{H}{U}\cong C_{p+1}$ or $|H:\cent{H}{U}|=2$ and $H\cong C_{2(p+1)}$ where $p>2$.
		\item[(b)] $H=\cent{H}{U}\cong Q_{(p^2-1)/d}$, where $p\in\mathbb{M}$ and $d\in\{(p-1)/2,p-1\}$.
	\end{description}
\end{lem}
\begin{proof}
	Write $V=P/U$.
	Since $U$ and $V$ are the $G$-chief factors of the nonabelian $p$-group $P$,
	$$U=\Phi(P)=P'=\z{P}$$
	where $|V|=|P/U|\geq p^2$,
	and hence $\cent{H}{P}=\cent{H}{V}$.
	Since nilpotent group $H$ acts Frobeniusly on $V$, $H$ is either cyclic or generalized quaternion by cyclic such that $\cent{H}{P}=\cent{H}{V}=1$.
	Notice that $U\leq P'\cap \z P$ and that $V$ is an $H$-module,
	and hence $U$ being an $H$-module is a homomorphic image of $V$'s Schur multiplier $V\wedge V$ (see, for instance, \cite[11.4.16]{robinson1996}).
	As $(|H|,|P|)=1$, Maschke's theorem yields that $U$ is a direct summand of $V\wedge V$.

	\smallskip

	\noindent\emph{Claim 1. If $H$ is cyclic, then $|\cent{H}{U}|=(p^m+1,|H|)$ for some integer $m$ such that $1\leq m<n$ and
		\begin{equation}\label{21}
			\frac{p^n-1}{d}~\Big|~(p^m+1,\frac{p^n-1}{d})^2.
		\end{equation}
	}
	\indent Let $\mathrm{K}$ be a splitting field of $H$ with characteristic $p$ and $H=\la h\ra$.
	Then $\cent{H}{U}=\cent{H}{U\otimes \mathrm{K}}$ and $\cent{H}{V}=\cent{H}{V\otimes \mathrm{K}}$ by \cite[Lemma 5.2]{doerk1992}; $U\otimes \mathrm{K}$ is a direct summand of $(V\otimes \mathrm{K})\wedge (V\otimes \mathrm{K})$ as $(V\otimes \mathrm{K})\wedge (V\otimes \mathrm{K})=(V\wedge V)\otimes \mathrm{K}$.
	Let $\o H=H/\cent{H}{U}$.
	Note that $U$ is a faithful irreducible $\la  \o h\ra$-module over $\mathbb{F}_p$ (the finite field of $p$ elements),
	and so $U\otimes \mathrm{K}$ is a faithful $\la  \o h\ra$-module over $\mathrm{K}$,
	thus by \cite[Theorem 9.21]{isaacs1994} the eigenvalues of $\o h$ on $U\otimes \mathrm{K}$ form a $\gal{\mathrm{K}}{\mathbb{F}_p}$-class.
	Let $\xi$ be a representative of the class. Then $o(\o h)=o(\xi)$.
	Recall that $V\cong (C_p)^n$, we write
	\[
		V\otimes \mathrm{K}=\bigoplus_{i=0}^{n-1} V_i
	\]
	where $V_i=\la v_i\ra$ are faithful $\la  h\ra$-submodules of $V\otimes \mathrm{K}$
	as $\cent{H}{V_i}=\cent{H}{V\otimes \mathrm{K}}$.
	Now, \cite[Theorem 9.21]{isaacs1994} implies that $v_i^{ h}=\varepsilon^{p^i}v_i$ where $o(\varepsilon)=o( h)$ for $0\leq i\leq n-1$, and so
	\[
		(v_i\wedge v_j)^{h}=\varepsilon^{p^i+p^j}v_i\wedge v_j.
	\]
	Observe that $(V\otimes \mathrm{K})\wedge (V\otimes \mathrm{K})$ is spanned by $v_i\wedge v_j$'s for $0\leq i<j\leq n-1$
	and that $\xi$ is one of these $\varepsilon^{p^i+p^j}$, there exist $i_0$ and $j_0=i_0+m$ such that
	$\xi=\varepsilon^{p^{i_0}+p^{j_0}}=\varepsilon^{p^{i_0}(p^m+1)}$.
	Therefore
	$$|H:\cent{H}{U}|=o(\o h)=o(\xi)=\frac{o(\varepsilon)}{(p^m+1,o(\varepsilon))}=\frac{|H|}{(p^m+1,|H|)},$$
	and so $|\cent{H}{U}|=(p^m+1,|H|)$.

	As cyclic group $\o H$ acts faithfully and irreducibly on $U$, the action is a Frobenius action,
	and so $\cent{H}{U}=\cent{H}{u}$ for each nontrivial $u\in U$.
	Since $|H|=(p^n-1)/d$, where $d\mid p-1$, divides $|\cent{H}{u}|^2$ for each $u\in U$,
	$|H|\mid |\cent{H}{U}|^{2}$,
	and hence (\ref{21}) follows immediately.

	\smallskip

	\noindent\emph{Claim 2. If $|P:U|=p^2$, then $P\cong E(p^{1+2})$. In particular, $\cent{H}{U}=\cent{H}{u}$ for each nontrivial element $u$ in $U$.}

	\smallskip

	Under the current assumption, as $U=\Phi(P)=P'=\z{P}$, $P$ is an extraspecial $p$-group of order $p^3$ by Lemma \ref{specialextra}.
	Also since $U$ and $P/U$ are the $G$-chief factor of $P$, we conclude that $P\cong E(p^{1+2})$.
	As $U\cong C_p$ is a faithful $H/\cent{H}{U}$-module, $H/\cent{H}{U}$ is cyclic and it acts Frobeniusly on $U$, so $\cent{H}{U}=\cent{H}{u}$ for each nontrivial element $u$ of $U$.

	\smallskip

	\noindent\emph{Claim 3. If Zsigmondy prime divisor of $p^n-1$ exists, then $P\cong E(p^{1+2})$ and $(a)$ holds.}

	\smallskip

	Since $H\cong HU/U$, an application of Lemma \ref{Frobcomp} to $G/U$ yields that
	$H\cong C_{(p^n-1)/d}$.
	Let $q$ be a Zsigmondy prime divisor of $p^n-1$.
	Recall that $|V|\geq p^{2}$, and so $q\mid (p^n-1)/d$ as $d\mid p-1$.
	By (\ref{21}), $q\mid p^m+1$.
	Note that $n$ is the smallest positive integer such that $q\mid p^n-1$, and hence we conclude from
	$q\mid p^{2m}-1$ that $n= 2m$ as $m<n$.
	Consequently, $|\cent{H}{U}|=p^m+1$ and again by (\ref{21})
	\[
		\frac{p^m-1}{d}~\Big|~p^m+1.
	\]
	Since $(p^m-1,p^m+1)\mid 2$, one has $(p^m-1)/d\mid 2$, and so $p^m-1 \mid 2(p-1)$.
	Therefore $m=1$ and $d\in\{p-1,(p-1)/2\}$.
	Recall that $\dim_{\mathbb{F}_p}(V)=n=2m=2$ and that $U$ is a nontrivial summand of $V\wedge V$ as an $\mathbb{F}_p$-space,
	and then
	$$1\leq \dim_{\mathbb{F}_p}(U)\leq \dim_{\mathbb{F}_p}(V\wedge V)=\frac{n(n-1)}{2}=1.$$
	Observe that $|P:U|=|V|=p^{2}$, and so
	$P\cong E(p^{1+2})$ by Claim 2.
	Furthermore, if $d=p-1$, then $H=\cent{H}{U}\cong C_{p+1}$;
	if $d=(p-1)/2$, then $H\cong C_{2(p+1)}$ such that $|H:\cent{H}{U}|=2$ where $p>2$.

	\smallskip

	\noindent\emph{Claim 4. If Zsigmondy prime divisor of $p^n-1$ doesn't exist and $H$ is cyclic, then $P\cong E(p^{1+2})$ and $(a)$ holds.}

	\smallskip

	Since $H\cong HU/U$, another application of Lemma \ref{Frobcomp} to $G/U$ yields that
	either $H\cong C_{(p^2-1)/d}$ where $p\in \mathbb{M}$ or $H\cong C_{63}$ where $p^n=2^6$.
	Assume the latter.
	Then (\ref{21}) fails, a contradiction.
	So $H\cong C_{(p^2-1)/d}$ where $p\in \mathbb{M}$.
	As $|P:U|=|V|=p^2$, $P\cong E(p^{1+2})$ by Claim 2.
	By (\ref{21}),
	\[
		\frac{p-1}{d}~\Big|~ p+1,
	\]
	so it is routine to check that (a) holds.

	\smallskip

	\noindent\emph{Claim 5. If Zsigmondy prime divisor of $p^n-1$ doesn't exist and $H$ is noncyclic, then $P\cong E(p^{1+2})$ and $(b)$ holds.}

	\smallskip

	Note that $H\cong HU/U$.
	Another application of Lemma \ref{Frobcomp} to $G/U$ yields that $V\cong (C_p)^2$ and $H=Q\times D$ where $Q\in\syl{2}{H}$ is generalized quaternion and $D$ is either isomorphic to $C_{(p-1)/d}$ or $C_{(p-1)/2d}$.
	Since $|P:U|=|V|=p^2$, $P\cong E(p^{1+2})$ by Claim 2.
	Write $D=\la g\ra$.
	Since $o(g)\mid p-1$ and $V\cong (C_p)^2$,
	$$V=V_1\oplus V_2,$$
	where $V_i$ are irreducible $\la g\ra$-submodules of $V$.
	As $H=Q\times D$, $V$ is a homogenous $\la g\ra$-module by Clifford's theorem, so $g$ acts as scalar multiplication by $\varepsilon\in \mathbb{F}_p$ on $V$.
	Observe that $P$ is an extraspecial $p$-group of order $p^3$,
	and so $g$ acts as scalar multiplication by $\det (g)=\varepsilon^2$ on $U$.
	Since $\cent{\la g\ra}{V}=1$ and $o(g)$ is odd, $o(g)=o(\varepsilon)=o(\varepsilon^2)$,
	and hence $\cent{D}{U}=1$.
	Note that $|D|$ is a Hall number of $|G|$ such that
	$$|D|=|D:\cent{D}{U}|~\big|~|H:\cent{H}{U}|$$
	and that $|H|\mid |\cent{H}{U}|^2$ by Claim 2 as $P\cong E(p^{1+2})$,
	and then $|D|^2~\big|~ |G|$.
	Therefore $D=1$ and hence $H=Q$.
	Furthermore, either $d=p-1$ and $H\cong Q_{p+1}$ or $d=(p-1)/2$ and $H\cong Q_{2(p+1)}$.

	To reach the final conclusion of (b), it suffices to show that $H=\cent{H}{U}$.
	For $p\in \mathbb{M}$, observe that every irreducible cyclic subgroup with order dividing $p+1$ in $\GL(2,p)$ is contained in $\SL(2,p)$,
	and hence elements of order $p+1$, $(p+1)/2$ or $4$ (as none of the orders divide $p-1$) lie in $\SL(2,p)$.	
	Consequently, subgroups of $\GL (2,p)$ which are isomorphic to $Q_{p+1}$ or $Q_{2(p+1)}$ are contained in $\SL(2,p)$.
	As $H$ acts Frobeniusly on $V=P/\Phi(P)$, $H$ is isomorphic to a subgroup of $\Out (P)$.
	Since $P\cong E(p^{1+2})$ where $p\in\mathbb{M}$,
	$$\Out (P)=\cent{\Out(P)}{U}\rtimes C\cong \GL(2,p)$$
	where $\cent{\Out(P)}{U}\cong \SL(2,p)$ and $C\cong C_{p-1}$ acts Frobeniusly on $U$.
	So $H=\cent{H}{U}$ as $H$ being isomorphic to either $Q_{p+1}$ or $Q_{2(p+1)}$.
\end{proof}

\begin{lem}\label{W=1}
	Let a nontrivial group $H$ act faithfully and coprimely on a $p$-group $P$ and let $G=P\rtimes H$.
	Let $1\lhd  W\lhd U\lhd  P$ be a $G$-chief series of $P$ such that $U/W$ is cyclic.
	Assume that
	$U=P'$ is the unique maximal $G$-invariant subgroup of $P$ such that $|P:U|=p^2$ and that $P/U$ is $\cent{H}{U/W}$-irreducible.
	Then $p>2$, and $W$ is isomorphic to $P/U$ as a $\cent{H}{U/W}$-module.
\end{lem}
\begin{proof}
	Note that $1\lhd W\lhd U=P'\lhd P$ is a $G$-chief series of $P$.
	Then $W\leq P'\cap\z{P}$ and hence $W$ is isomorphic to a quotient of the Schur multiplier $\mathrm{M}(P/W)$.

	Now, we claim that $p>2$. In fact, if $p=2$, then $P/W$ is a special $p$-group such that $|P/W:(P/W)'|=4$; by Lemma \ref{specialextra} $P/W\cong Q_8$ as $P'$ being the unique maximal $G$-invariant subgroup of $P$; this implies a contradiction $1<|W|\le |\mathrm{M}(P/W)|=|\mathrm{M}(Q_8)|=1$.

	We next show that $W=\z{P}$.
	Otherwise, the uniqueness of $U=P'$ yields $U=\Phi(P)=P'=\z P$.
	Since $|P:U|=p^2$, by Lemma \ref{specialextra} $U=P'\cong C_p$, a contradiction.

	So $W=\z P$, and hence $W=\z P=[P,P']=[P,U]$.
	Notice that $U/W$ is cyclic, and then $U$ is abelian.
	Let $C=\cent{H}{U/W}$.
	Fix an element $u $ of $U-W$ with $\la uW\ra=U/W$,
	and define a map $\sigma:P/U\rightarrow W$ by setting $\sigma(x)=[x,u]$ for $x\in P/U$.
	Since $W=\z P=[P,U]$ and $U/W=\la uW\ra$,
	it is routine to check that $\sigma$ is a $C$-epimorphism.
	As $P/U$ is $C$-irreducible, $\sigma$ is a $C$-isomorphism by Schur's lemma.
\end{proof}

\section{Proof of the main theorem}

If $G$ is a finite group of order dividing $n$, then the Galois group $\mathfrak{G}=\gal{\mathbb{Q}_n}{\mathbb{Q}}$
of the $n$-th cyclotomic extension naturally acts on $\irr{G}$ via
\[\chi^{\mathfrak{g}}(g)=\chi(g)^{\mathfrak{g}}\]
for $\chi\in\irr G$, $\mathfrak{g}\in\mathfrak{G}$, and $g\in G$.
We say two characters $\alpha,\beta\in\irr G$ are Galois conjugate if there exists some $\mathfrak{g}\in\mathfrak{G}$
such that $\alpha=\beta^{\mathfrak{g}}$.
And $\alpha,\beta\in\irr{G}$ are Galois conjugate if and only if
there exists an $\mathfrak{a}\in\gal{\mathbb{Q}(\alpha)}{\mathbb{Q}}$ such that $\alpha=\beta^{\mathfrak{a}}$
where $\mathbb{Q}(\alpha)$ is \emph{the field of values} of $\alpha$.
We also note that if $\alpha,\beta \in \irr{G}$ are Galois conjugate, then $\mathbb{Q}(\alpha )=\mathbb{Q}(\beta)$ and $\ker(\alpha)=\ker(\beta)$, i.e., they share the same field of values and
kernel.

\begin{pro}\label{sol}
	Let $G$ be a group such that $\irrs{G}$ forms a Galois conjugacy class. Then $G$ is solvable.
\end{pro}
\begin{proof}
	Let $\chi_0\in\irrs{G}$ and $K=\ker(\chi_0)$.
	As $\irrs{G}$ forms a Galois conjugacy class, $\ker(\chi)=K$ for $\chi\in\irrs{G}$.
	We now proceed by induction to show the solvability of $G$.
	Notice that for each quotient $G/N$ of $G$, either $\irrs{G/N}=\varnothing$ or $\irrs{G/N}$ forms a Galois conjugacy class.
	So, by \cite[Theorem A]{gagolalewis1999} and induction, we may assume that $G$ admits the unique
	minimal normal subgroup, say $E$, such that $G/E$ is solvable.
	Also, we further assume that $E=E_1\times \cdots\times E_t$ where $E_i\cong S$ for a nonabelian simple group $S$.
	We next show the assumption that $S$ is nonabelian simple implies a contradiction.

	Assume that $S\notin\{A_7,A_{11},A_{13},M_{22}\}$.
	Let $p\in\pi(S)$.
	By \cite[Lemma 2.4]{qian2012 p-closed}, there exists some $\varphi_{pi}\in\irr {E_i}$ such that $(\varphi_{pi}(1)^2)_p>|\Aut(E_i)|_p$.
	Since $G$ acts transitively on $\{E_1,\dots,E_t\}$, we may assume that $\varphi_{p1},\dots,\varphi_{pt}$ are $G$-conjugate.
	Write $\psi_p=\varphi_{p1}\times \cdots\times \varphi_{pt}$.
	Then $\psi_p\in\irr E$.
	Applying \cite[Lemma 2.6]{qian2012 p-closed} to $G$, one has that there exists a faithful character $\chi_p\in \irr{G}$ lies over $\psi_{p}$ such that $\chi_p(1)^2\nmid |G|$.
	That is $\chi_p\in\irrs{G}$ for each $p\in\pi(S)$.
	Observe that $\irrs{G}$ forms a Galois conjugacy class.
	Hence, each pair of characters in $\{\chi_p:p\in\pi(S)\}$ are Galois conjugate and therefore every character in
	$\{\psi_{p}:p\in\pi(S)\}$ shares the same degree.
	Since the $q$-part of the integer $\psi_{p}(1)^2$ exceed $|E|_q$ for all $q\in \pi(S)$,
	we conclude that $\psi_{p}(1)^2>|E|$ which contradicts $\psi_p\in\irr E$.

	Assume that $S\in\{A_7,A_{11},A_{13},M_{22}\}$.
	Let $\ell=\max(\pi(S))$.
	By checking \cite{atlas}, $E_{i}$ has two $\ell$-defect zero irreducible characters $\alpha_{i}$ and $\beta_{i}$ which have distinct degrees.
	Let $\alpha=\alpha_1\times\cdots\times\alpha_t$ and $\beta=\beta_1\times \cdots\times \beta_t$, and let $\chi,\omega\in\irr{G}$ be characters lie over $\alpha$ and $\beta$ respectively.
	Then $\chi(1)_\ell\geq |E|_\ell$ and $\omega(1)_\ell\geq |E|_\ell$.
	Since by \cite[Proposition 2.5]{qian2012 p-closed} $|E|_\ell>|G/E|_\ell$,
	$\chi(1)^2\nmid |G|$ and $\omega(1)^2\nmid |G|$,
	and so $\chi,\omega\in\irrs{G}$.
	However, $\alpha(1)\neq \beta(1)$, a contradiction.
\end{proof}

By Lemma \ref{sol}, it remains to classify solvable groups $G$ in which $\irrs{G}$ forms one Galois conjugacy class.
For convenience, we introduce the next hypothesis.

\begin{hy}\label{hy}
	Let $G$ be a solvable group and $P=G^{\infty}$ its nilpotent residue,
	and let $\chi_0\in\irrs{G}$, $K=\ker(\chi_0)$ and $U=P\cap K$.
	Assume that $\mathfrak{G}=\gal{\mathbb{Q}(\chi_0)}{\mathbb{Q}}$ acts transitively on $\irrs{G}$.
\end{hy}

Let $G$ be a group and $N$ its normal subgroup, and let $\theta\in \irr{N}$. We use $\irr{G|\theta}$ to denote the set of irreducible characters of $G$ lying over $\theta$ and $\irr{G|N}$ to denote the complement of the set $\irr{G/N}$ in the set $\irr{G}$.

\begin{lem}\label{roughstructureofG}
	Assume the setting of Hypothesis \ref{hy}.
	Then $G=P\rtimes H$ where $P\in \syl{p}{G}$ and $H$ is nilpotent subgroup of $G$; $K=\cent{H}{P}\times U$ is nilpotent where $U=\Phi(P)=P'$ is the unique maximal $G$-invariant subgroup of $P$.
\end{lem}
\begin{proof}
	Since $\mathfrak{G}$ acts transitively on $\irrs{G}$, every character in $\irrs{G}$ shares the same degree and kernel.
	Note that $G/N$ is nonnilpotent for every proper $G$-invariant subgroup $N$ of $P$ as $P=G^{\infty}$,
	and then there exists an $\omega\in\irrs G$ such that $N\leq P\cap \ker(\omega)=P\cap K=U$.
	So, $U$ is the unique maximal $G$-invariant subgroup in $P$.

	We now claim that $P$ is a Sylow $p$-subgroup of $G$.
	If not, let $G$ be a counterexample of minimal order.
    Assume first that $U>1$.
	Notice that $P/U=G^{\infty}/U=(G/U)^\infty$ and that $\irrs{G/U}$ forms a Galois conjugacy class,
	and so the minimality of $G$ implies that $P/U\in \syl{p}{G/U}$, in particular, $|G:P|$ is a $p'$-number.
    Since $P$ possesses the unique maximal $G$-invariant subgroup $U$, $P$ is nonnilpotent such that
	$$1<\fitt P\leq U<P.$$
	Observe that $G/\fitt{P}$ and $P/\fitt{P}$ satisfies Hypothesis \ref{hy},
	and hence $P/\fitt{P}\in \syl{p}{G/\fitt{P}}$ by the minimality of $G$.
	Also
	\[
		P/\Phi(P)\cong \fitt{P}/\Phi(P)\rtimes P/\fitt{P},
	\]
	where $\fitt{P}/\Phi(P)$ is a faithful completely reducible $P/\fitt{P}$-module (possibly of mixed characteristic).
	As $P$ is nonnilpotent, $\fitt{P}/\Phi(P)$ is not a $p$-group.
	If $p\mid|\fitt{P}/\Phi(P)|$, then $\fitt{P}/\Phi(P)=A/\Phi(P)\times B/\Phi(P)$
	where $A$ and $B$ are distinct $G$-invariant subgroups of $U$,
    and then $P/A$ and $P/B$ are $p$-groups by the minimality of $G$. 	
    Therefore, $P/\Phi(P)$ is a $p$-group as $P/\Phi(P)$ can be embedded in $P/A\times P/B$, a contradiction.
	Notice that $P/\fitt{P}$ is a $p$-group,
	and then $\fitt{P}/\Phi(P)\in\mathrm{Hall}_{p'}(P/\Phi(P))$.
	So, by Lemma \ref{isaacs} there exists a $\theta\in\irr{\fitt{P}/\Phi(P)}$ such that
	\[
		|P:\inert{P}{\theta}|>|P/\fitt{P}|^{\frac{1}{2}}.
	\]
	Take $\omega\in\irr{G/\Phi(P)\mid\theta}$, and note that $|P:\inert {P} {\theta} |$ is a $p$-number as $\fitt{P}\le \inert {P} {\theta} $, so
	\[
		(\omega(1)_p)^2 \geq (|G:\inert G \theta  |_p)^{2}\geq |P:\inert{P}{\theta}|^2>|P/\fitt{P}|=|G|_p,
	\]
	where the second inequality holds as $P \unlhd G$, and $\fitt{P}\nleq \ker(\omega)$, a contradiction to $\ker(\omega)=K \geq \fitt{P}$.
    So, we may assume that $U=1$.
	In other words, $P$ is minimal normal in $G$. 
	As $P=G^{\infty}$, $P \cap  \Phi(G)=1$, and hence $G=P\rtimes H$ where $H$ is nilpotent.
	Let $D\in\syl{p}{H}$ and let $E$ be its normal complement in $H$.
	Then $G=(P\rtimes E)\times D$ where $G/D\cong PE$ is nonnilpotent, so $D \leq K$, $P$ is minimal normal in $PE$ and
	there exists a character $\varphi\in \irrs{PE}$ by \cite[Theorem A]{gagolalewis1999}.
	Furthermore, $\ker(\varphi)\le E$ (as $\ker(\varphi)\cap P=1$) and $\varphi(1)$ is a $p'$-number (as $P$ being abelian normal in $PE$) such that $\varphi(1)^2\nmid |E:\ker(\varphi)|$.
	Let $\mu\in \irr{D/D'}$, to conclude a contradiction, it suffices to show $\mu=1_D$ as $D$ being a nontrivial $p$-group.
	To see this, let $\chi=\varphi\times \mu$, and note that $\chi$ is an irreducible character of $G$ such that $\chi(1)=\varphi(1)$ and $\ker(\chi )\cap PE=\ker(\varphi)$, and so
	by calculation $\chi\in \irrs{G}$, and this implies that $\mu=1_D$ because $D\leq K=\ker(\chi)$, as required.

	Thus $P\in\syl{p}{G}$ and so $G=P\rtimes H$ where $H\in\mathrm{Hall}_{p'}(G)$ is nilpotent.
	Let $C=H \cap K$.
	We next show that $K=C\times U$ is nilpotent where $U=\Phi(P)$, and $C=\cent{H}{P}$.
	Observe that $K\unlhd G$ and that $P\in\syl{p}{G}$, and then $U=P\cap K\in\syl pK$ and $C=H\cap K\in\mathrm{Hall}_{p'}(K)$,
	i.e. $K=U\rtimes C$ where $U\in\syl p K$.
	Clearly, $\Phi(P)\leq U$.
	Since $H$ acts coprimely on $P/\Phi(P)$, $U=\Phi(P)$ by the uniqueness of $U$.
	Since
	$$[C,P]\leq [K,P]\leq K\cap P=U=\Phi(P),$$
	$[C,P/\Phi(P)]=1$.
	As $C$ acts coprimely on $P$, $[C,P]=1$ and $K=C\times U$ is nilpotent.
	Since $K$ is nilpotent with $C$ being its Hall $p'$-subgroup,
	one has $C\unlhd G$.
	To see $C=\cent{H}{P}$, it remains to show that $\cent{H}{P}\leq K$ as $C=H\cap K$ and $[C,P]=1$.
	Recall that $P=G^{\infty}$ and $\cent{H}{P}\cap P=1$, and hence $G/\cent{H}{P}$ is a nonnilpotent quotient of $G$.
	It follows by \cite[Theorem A]{gagolalewis1999} that there exists a $\chi\in\irrs{G/\cent{H}{P}}$.
	Thus $\cent{H}{P}\leq \ker(\chi)=K$.

	Finally, we show that $U=P'$.
	If not, let $G$ be a minimal counterexample.
	Notice that for each $G$-invariant proper subgroup $N$ of $U$ such that $G/N$ is nonnilpotent, $G/N$ satisfies the hypothesis of this lemma.
	So, to get a contradiction, we may assume that $P'=1$ and also that $U$ is minimal normal in $G$.
	That is
	$$U=\Phi(P)=\Omega_1(P),$$
	where $\Omega_{1}(P)=\{x\in P:x^p=1\}$.
	Since $G/\cent{H}{P}$ also satisfies the hypothesis of this lemma, we may keep assuming that $\cent{H}{P}=1$.
	As $\cent{H}{P/U}=\cent{H}{P/\Phi(P)}=\cent{H}{P}=1$, by Lemma \ref{isaacs} there exists a $\lambda\in\irr{P/U}$ such that
	\[
		|H:\inert{H}{\lambda}|^2\nmid |H|.
	\]
	Observe that the $p$-power map is an $H$-isomorphism from $P/U$ to $U$.
	Therefore, there exists a nonprinciple character $\mu\in\irr{U}$ such that
	$\inert{H}{\mu}=\inert{H}{\lambda}$.
	Let $\chi\in\irr{G|\lambda}$ and $\omega\in\irr{G|\mu}$.
	Since $P\in\syl{p}{G}$ is abelian normal in $G$, $\chi(1)$ and $\omega(1)$ are both coprime to $p$.
	By Clifford's theorem and calculation, $\chi,\omega\in\irrs{G}$.
	However, as $U\leq\ker(\chi)$ and $U\cap \ker(\omega)=1$, $\ker(\chi)\neq \ker(\omega)$,
	a contradiction.
\end{proof}

Let $V$ be an elementary abelian $p$-group and $\mathfrak{T}=\gal{\mathbb{Q}_p} {\mathbb{Q}} $, and let $\xi=e^{2\pi\sqrt{-1}/p}$.
It is well-known that $\mathfrak{T}= \la \mathfrak{t}\ra\cong C_{p-1}$ such that $\xi ^{\mathfrak{t}}=\xi ^{k}$ where $k$ is a primitive element of $\mathbb{F}_{p}$.
For $v\in V$, $v^\mathfrak{t}$ is defined to be the element of $V$ such that
\[
	\alpha(v^\mathfrak{t})=\alpha^\mathfrak{t}(v)=(\alpha(v))^\mathfrak{t}
\]
for all $\alpha\in \irr{V}$.
Since $(\alpha(v))^\mathfrak{t}=(\alpha(v))^{k}$ for all $\alpha\in \irr{V}$ and $v\in V$, we have that $\alpha^\mathfrak{t}=\alpha^{k}$ and $v^{\mathfrak{t}}=v^{k}$ for all $\alpha\in \irr{V}$ and $v\in V$,
i.e. $\mathfrak{T}=\z{\GL (V)}$ when we identify $\mathfrak{T}$ as a subgroup of $\GL(V)$.

Recall that for a group $G$, $G^\sharp$ stands for the set of nontrivial elements of $G$.

\begin{lem}\label{K=1}
	Assume the setting of Hypothesis \ref{hy}.
	If $K=1$, then $G=V\rtimes H$ is a Frobenius group such that kernel $V\cong (C_p)^n$ is minimal normal in $G$ and complement $H$ is nilpotent of order $(p^n-1)/d$ where $d=|\irrs{G}|$ divides $p-1$;
	the field of values $\mathbb{Q}(\chi_0)$ is a subfield of $\mathbb{Q}_p$.
	Furthermore, if we identify $H$ as a subgroup of $\GL(V)$, then $H\z{\GL(V)}$ acts transitively on $V^\sharp$.
\end{lem}
\begin{proof}
	Since $K=1$, it follows by Lemma \ref{roughstructureofG} that $G=V\rtimes H$ where $V\in\syl{p}{G}$ is the unique minimal normal subgroup of $G$ and $H$ is nilpotent.
	Let $\lambda$ be a nonprinciple irreducible character of $V$, and write $T=\inert{G}{\lambda}$ and $I=T\cap H$.
	Note that $V$ is an abelian normal subgroup of $G$, and hence $\chi(1)\mid |H|$ for $\chi\in\irr{G|\lambda}$. 
	We now argue that either $|H|\mid |I|^{2}$ or $I=1$.
	
	If there exists a $\chi\in \irr{G|\lambda}\cap\irrn{G}$, then $|H:I|^2=|G:T|^2\mid \chi(1)^2\mid |H|$,
	that is $|H|\mid |I|^2$.
    Suppose now $\irr{G|\lambda}\subseteq \irrs{G}$.
	We argue that $T=V$, and so $I=1$.
	Since $(|V|,|T/V|)=1$, by \cite[Corollary 6.27]{isaacs1994} $\lambda$ has a canonical extension $\mu\in \irr{T|\lambda}$ such that $o(\mu)=o(\lambda)$.
	Let $\beta\in\irr{T/V}$, $\omega=\mu^G$ and $\chi=(\mu\beta)^G$.
	By Gallagher correspondence, $\chi,\omega\in \irr{G|\lambda}$.
	As $\chi,\omega\in\irr{G|\lambda} \subseteq \irrs{G}$, there exists an $\mathfrak{a}\in\mathfrak{G}$ such that $\chi=\omega^{\mathfrak{a}}$.
	So $T/V$ is abelian.
	Observe that $\lambda$ and $\lambda^{\mathfrak{a}}$ both are irreducible constituents of $\chi_V$,
	and hence $\lambda^g=\lambda^{\mathfrak{a}}$ for some $g\in G$ by Clifford's theorem.
	Since $\mu^{\mathfrak{a}g^{-1}}\in\irr{T|\lambda}$ such that
	\begin{center}
		$(\mu^{\mathfrak{a}g^{-1}})^G=(\mu^{\mathfrak{a}})^G=(\mu^G)^{\mathfrak{a}}=\omega^{\mathfrak{a}}=\chi$,
	\end{center}
	it follows by Clifford correspondence that $\mu^{\mathfrak{a}g^{-1}}=\mu\beta$.
	By comparing the order of the two linear characters, we have that
	\[
		o(\mu)=o(\mu^{\mathfrak{a}g^{-1}})=o(\mu\beta)=o(\mu)o(\beta)
	\]
	where the last equality holds as $o(\mu)=o(\lambda)$ being coprime to $o(\beta)$.
	So $o(\beta)=1$ and hence $T=V$, as desired.
	Recall that nilpotent group $H$ acts faithfully and irreducibly on $V$.
	An application of Lemma \ref{Frob} to $G$ yields that $G$ is a Frobenius group with kernel $V$ being minimal normal in $G$.

	Now take $\chi\in\irrs{G}$, then $\chi=\lambda^G$ for some nonprinciple character $\lambda\in\irr{V}$.
	So
	\begin{equation}\label{field}
		\mathbb{Q}(\chi_0)=\mathbb{Q}(\chi)=\mathbb{Q}(\lambda^G)\subseteq\mathbb{Q}(\lambda)=\mathbb{Q}_p
	\end{equation}
	where $\mathbb{Q}_p$ is the $p$-th cyclotomic field.
	Write $\mathfrak{T}=\gal{\mathbb{Q}_p}{\mathbb{Q}}$.
	By (\ref{field}),
	$\mathfrak{G}=\{\mathfrak{t}_{\mathbb{Q}(\chi_0)}:\mathfrak{t}\in\mathfrak{T}\}$
	where $\mathfrak{t}_{\mathbb{Q}(\chi_0)}$ is the restriction of $\mathfrak{t}$ on $\mathbb{Q}(\chi_0)$.
	Since $\irrs{G}$ forms one $\mathfrak{G}$-class,
	$\mathfrak{T}$ acts transitively on
	$$\irrs{G}=\{\lambda^G:\lambda\in\irr{V|V}\}.$$
	So $|\irrs{G}|~\big|~ |\mathfrak{T}|$ where $\mathfrak{T}\cong C_{p-1}$ by Galois theory.
	Write $d=|\irrs{G}|$ and recall that $V\cong(C_p)^n$.
	We conclude that $|H|=(p^n-1)/d$ where $d\mid p-1$.
	Since $\mathfrak{T}\cong C_{p-1}$ acts as scalar multiplication on $\irr{V}$,
	by the discussion preceding this lemma, it also acts as scalar multiplication on $V$,
	besides we may identify $\mathfrak{T}$ as $\z{\GL(V)}$.
	In addition, if we identify $H$ as a subgroup of $\GL(V)$, then $H\z{\GL(V)}$ acts transitively on $V^\sharp$.
\end{proof}




\begin{lem}\label{C>1U>1}
	Assume the setting of Hypothesis \ref{hy}. Let $G$ be the group described in Lemma \ref{roughstructureofG}. Then the following hold.
	\begin{description}
		\item[(a)] If $U\neq K$, then $H\in\syl{q}{G}$ for some prime $q$ such that $H/\cent{H}{P}\cong C_q$.
		\item[(b)] If $U=K>1$, then $P$ and $H$ are the groups listed in Lemma \ref{exspec}.
	\end{description}
\end{lem}
\begin{proof}
	Write $C=\cent{H}{P}$. Then $C=H\cap K$ by Lemma \ref{roughstructureofG}.

	Suppose that $U\neq K$.
	Let $K/E$ be a $G$-chief factor such that $E\ge U$.
	We proceed by induction on $|G|$ to show that $|H/C|=|K/E|=q$.
	Notice that $\irrs{G/E}$ forms a Galois conjugacy class as $\chi_0\in \irrs{G/E}$,
	and then by Lemma \ref{roughstructureofG} $G/E$ satisfies the hypothesis of this lemma.
	Also since
	\[
		HE/E \cap K/E = CE/E,~~|HE/E:CE/E|=\frac{|H/(H\cap E)|}{|C/(C\cap E)|}=|H:C|,
	\]
	where the last equality of the second formula holds as $H \cap  E=H \cap K \cap  E=C \cap  E$,
	and note that $\cent{HE/E}{P/E}=CE/C$,
	we may assume that $E=1$ by induction.
	Then $K=C$ is minimal normal in $G$ and hence $K$ is minimal normal in $H$ as $[K,P]=1$.
	As $H$ is nilpotent, $K\leq \z{H}$, this implies $K\leq \z G$, and hence
	$K\cong C_q$ for some prime $q$.
	To reach our goal, it remains to show $|H/K|=q$ as $K=C$.
	Let $\theta\in \irr{K|K}$, $\lambda\in\irr{P|P}$ and $\varphi=\theta\times \lambda$,
	and observe that $\varphi\in\irr{K\times P}$ with $\varphi(1)=1$.
	Since $H/K$ acts Frobeniusly on $P$ by Lemma \ref{K=1}, $\inert{G}{\lambda}=KP$.
	Also since $K\leq\z{G}$,
	\[
		\inert{G}{\varphi}=\inert{G}{\lambda}=KP.
	\]
	So, Clifford correspondence yields that $\varphi^G\in\irr{G}$.
	Observing that
	$$\ker(\varphi^G)\cap KP\leq \ker(\theta\times \lambda)=\ker(\theta)\times \ker(\lambda)=1$$
	(the first equality holds as $(|K|,|P|)=1$)
	where $K=\ker(\chi_0)$ and $\chi_0\in\irrs{G}$,
	we have that $\varphi^G\in\irrn{G}$ as every character in $\irrs{G}$ sharing the same kernel.
	That is
	\begin{equation}\label{eq2}
		|H/K|^2=|H:\inert{H}{\varphi}|^2=|H:\inert{H}{\varphi}|^2\varphi(1)^2=(\varphi^G(1))^2~\big|~ |G:\ker(\varphi^G)|.
	\end{equation}
	Since $G/KP$ acts Frobeniusly on $P$, $\fitt{G}=KP$,
	and hence $\ker(\varphi^G)=1$ as $\ker(\varphi^G)\cap KP=1$.
	By (\ref{eq2}), $|H/C|=|H/K|=|K|=q$, as claimed.
	Thus $|K/E|=|H/C|=q$ for each $G$-chief factor $K/E$ such that $U\leq E$.
	As a consequence, $H\in\syl{q}{G}$ because $H$ is nilpotent.

	Suppose that $U=K>1$.
	Then $G/U$ is the group described in Lemma \ref{K=1}, in particular, $H$ acts Frobeniusly on $P/U$.
	We now argue that $U$ is minimal normal in $G$.
	Let $G$ be a minimal counterexample with $U$ not being minimal normal in $G$.
	Then there exists a $G$-chief series $1\lhd W\lhd U\lhd P$ of $P$.
	Also, by Lemma \ref{roughstructureofG} $U=\Phi(P)=P'$ is the unique maximal $G$-invariant subgroup of $P$.

	Let $\o G=G/W$ and let $\mu\in \irr {\o U}$ and $\chi\in\irr{\o G|\mu}$.
	We show that $|\o H|\mid |\inert{\o H}{\mu}|^2$.
	Note that if $\mu=1_{\o U}$, then this is automatic true, and so we may assume $\mu\neq 1_{\o U}$.
	Since $\ker(\chi)\neq \o K$, $\chi\in\irrn{\o G}$, and hence $\chi(1)^2\mid |\o G|$.
	As $\o U$ is central in $\o P$, $\mu$ is $\o P$-invariant.
	An application of Clifford's theorem to $\o G$, $\o U$, $\chi$ and $\mu$ yields that
	\[
		|\o H:\inert{\o H}{\mu}|^2=|\o G:\inert{\o G}{\mu}|^2~\big|~ \chi(1)^2~\big|~ |\o G|,
	\]
	and hence $|\o H|\mid |\inert{\o H}{\mu}|^2$ for each $\mu\in\irr{\o U}$, as claimed.
	As $(|\o H|,|\o P|)=1$, by \cite[Lemma 13.23]{isaacs1994} $|\o H|\mid |\cent{\o H}{\o u}|^2$ for each $\o u\in \o U$.

	Applying Lemma \ref{exspec} to $G/W$, one has that $P/W\cong E(p^{1+2})$ and
	\begin{equation}\label{eq3}
		|H:\cent{H}{U/W}|~\big|~ 2,~~~~p+1~\big|~ |\cent{H}{U/W}|.
	\end{equation}
	As $\cent{H}{U/W}$ acts irreducibly on $P/U\cong (C_p)^2$ (because $p+1\mid |\cent{H}{U/W}|$) and $\cent{H}{P}=1$ (because $U=\Phi(P)$ and $\cent{H}{P/U}=1$),
	Lemma \ref{W=1} implies that $W$ is isomorphic to $P/U$ as $\cent{H}{U/W}$-module.
	Observe that $\cent{H}{U/W}$ acts Frobeniusly on $P/U$, and then it also acts Frobeniusly on $W$.
	For a nonprinciple irreducible character $\theta$ of $W$, $|\cent{H}{U/W}|\mid \omega(1)$ for every $\omega\in\irr{G|\theta}$.
	By (\ref{eq3}), $|\cent{H}{U/W}|^2\nmid |H|$, so $(\omega(1)_{p'})^2\nmid |H|$.
	That is $\omega\in\irrs{G}$ which contradicts $\ker(\omega)\neq U=K$.

	Finally, part (b) holds by Lemma \ref{exspec}.
\end{proof}

We prove Theorem A by proving the next two theorems.


\begin{thm}\label{thm1}
	The statement $(a)$ in Theorem A is true.
\end{thm}
\begin{proof}
	Suppose that $G$ is a group in which $\irrs{G}$ forms a Galois conjugacy class.
	By Proposition \ref{sol}, $G$ is necessarily solvable, and hence we may assume that $G$ satisfies the setting of Hypothesis \ref{hy}.
	Lemma \ref{roughstructureofG} yields that
	$$G=P\rtimes H,$$
	where $P\in \syl{p}{G}$ and $H$ is a nilpotent subgroup of $G$; $K=\cent{H}{P}\times U$ is nilpotent where $U=\Phi(P)=P'$ is the unique maximal $G$-invariant subgroup of $P$.
	Observe that $\cent{H}{P}=\cent{H}{P/U}$ as $U=\Phi(P)$ and $(|H|,|P|)=1$,
	and hence by Lemma \ref{K=1} $H/\cent{H}{P}$ acts Frobeniusly and irreducibly on $P/U\cong (C_p)^n$ where $|H/\cent{H}{P}|=(p^n-1)/d$ such that $d\mid p-1$;
	and $\z{\GL(P/U)}(H/\cent{H}{P})$ acts transitively on $(P/U)^\sharp$ when we consider $H/\cent{H}{P}$ as a subgroup of $\GL(P/U)$; and $\mathbb{Q}(\chi)=\mathbb{Q}(\chi_0)\subseteq \mathbb{Q}_p$ for all $\chi\in \irrs{G}$. 
	Write $C=\cent{H}{P}$.

	Assume that $U=1$.
	Then $P\cong (C_p)^n$ and $G/C$ is a Frobenius group satisfies the hypothesis of Lemma \ref{Frobcomp}.
	Therefore, Lemma \ref{Frobcomp} implies that $H/C$ and $p$ satisfy either (a1) or (a2).
	If $C=1$, then we are done.
	Suppose that $C>1$.
	By part (a) of Lemma \ref{C>1U>1}, $H\in\syl{q}{G}$ for some prime $q$ such that $|H/C|=q$.
	Thus, $H/C$ satisfies (a1) as in (a2) $H/C$ has composite order.
	By calculation, $d=p-1$ and $q=1+p+\cdots+p^n$.
	So, (a3) holds.

	Assume that $U>1$.
	By part (b) of Lemma \ref{C>1U>1} and Lemma \ref{exspec}, $P\cong E(p^{1+2})$ and $U\cong C_p$; $H/C$ and $p$ satisfy one of (a4), (a5) and (a6).
	If $C=1$, then we are in one of the cases among (a4), (a5) and (a6).
	Suppose that $C>1$.
	By part (a) of Lemma \ref{C>1U>1}, $H\in\syl{q}{G}$ for some prime $q$ such that $|H/C|=q$.
	Thus, $H/C$ satisfies (a5) as $H/C$ must have composite order in (a4) or (a6).
	That is $q=p+1$.
	Therefore, $p=2$ and $q=3$ and so (a7) holds.
\end{proof}

\begin{thm}\label{thm2}
	The statement $(b)$ in Theorem A is true.
\end{thm}
\begin{proof}
	Suppose that $G$ is the group described in part (a) of Theorem A.
	Let $\chi_0 \in \irrs{G}$ and $K=\ker(\chi_0)$.
	Being the Hall $p'$-subgroup of a $G$-invariant subgroup $K$, $\cent{H}{P}\unlhd G$.
	Let $C=\cent{H}{P}$, and let $\tilde{G}=G/K$ and $\o G=G/C$.
	Note that $\mathbb{Q}(\chi_0)\subseteq \mathbb{Q}_p$.
	Write $\mathfrak{G}=\gal{\mathbb{Q}(\chi_0)}{\mathbb{Q}}$ and $\mathfrak{T}=\gal{\mathbb{Q}_p}{\mathbb{Q}}$.
	Then $\mathfrak{G}=\{\mathfrak{x}_{\mathbb{Q}(\chi_0)}:\mathfrak{x}\in\mathfrak{T}\}$.

	\smallskip

	\emph{Claim 1. $\irrs{\tilde{G}}$ forms a Galois conjugacy class.}

	\smallskip

	Without loss of generality, we may assume that $K=1$.
	Since $G=P\rtimes H$ is a Frobenius group with an elementary abelian kernel $P$ and a nilpotent complement $H$,
	and hence
	$$\irr{G|P}=\{\lambda^G:1_P\neq\lambda\in\irr{P}\}.$$
	Since $\mathfrak{T}$ acts as scalar multiplication on $\irr{P}$, by the discussion preceding Lemma \ref{K=1} it also acts as scalar multiplication on $P$.
	We now identify $H$ and $\mathfrak{T}$ as a subgroup of $\GL(P)$.
	Then $\mathfrak{T}=\z{GL(P)}$.
	Observe that $\z{GL(P)}H=\mathfrak{T}H$ acts transitively on $P^\sharp$,
	and hence it acts transitively on $\irr{P}^\sharp$ (since $\irr{P}$ has a group structure defined by the multiplication of characters of $P$) as $(|\z{GL(P)}H|,|P|)=1$.
	Since $\ker(\chi_0)=K=1$, $\chi_0\in\irr{G|P}$, and hence $\chi_0=\mu^G$ for some $\mu\in\irr{P|P}$.
	For each character $\lambda\in\irr{P|P}$,
	there exists some $h\in H$ and $\mathfrak{t}\in\mathfrak{T}$ such that $\lambda=\mu^{\mathfrak{t}h}$.
	Notice that $\mathfrak{G}=\{\mathfrak{x}_{\mathbb{Q}(\chi_0)}:\mathfrak{x}\in\mathfrak{T}\}$.
	There exists some $\mathfrak{g}\in\mathfrak{G}$ such that $\mathfrak{g}=\mathfrak{t}_{\mathbb{Q}(\chi_0)}$.
	Observe that
	$$\lambda^G=(\mu^{\mathfrak{t}h})^G=(\mu^G)^\mathfrak{t}=(\chi_0)^\mathfrak{t}=(\chi_0)^\mathfrak{g},$$
	we conclude that $\irrs{G}=\irr{G|P}$ forms a Galois conjugacy class as $G/P\cong H$ being nilpotent.

	\smallskip

	\emph{Claim 2. $\irrs{\o G}=\irrs{\tilde{G}}$.}

	\smallskip

	Without loss of generality, we may assume that $C=1$ and $U>1$.
	In this case, $U$ is the unique minimal normal subgroup of $G$.
	Also, $P\cong E(p^{1+2})$, $K=U\cong C_p$ and $H$ is one of the groups listed among (a4), (a5) and (a6).
	Let $\theta\in\irr{K}$ be nonprinciple.
	Since $\tilde{G}=G/K$, it suffices to show that $\chi\in\irrn{G}$ for each $\chi\in\irr{G|\theta}$.

	Let $M=P\rtimes \cent{H}{K}$.
	Then $M\unlhd G$ such that $|G:M|\mid 2$ and $K\le\z{M}$.
	Hence $M\leq \inert{G}{\theta}$.
	Since $P\cong E(p^{1+2})$, $\irr{P|\theta}=\{\varphi\}$ such that $\varphi_K=p\theta$ and $\theta^P=p\varphi$.
	The uniqueness of $\varphi$ yields $\inert{G}{\varphi}=\inert{G}{\theta}$.
	Observe that $(|H|,|P|)=1$, and so
	$\varphi$ is extendible to $\inert{G}{\varphi}=\inert{G}{\theta}$.
	Let $I=\inert{G}{\varphi}$ and $\psi\in\irr{I}$ be such that $\psi_P=\varphi$.
	Then by \cite[Corollary 6.17]{isaacs1994}
	\[
		\irr{I|\varphi}=\{\beta\psi:\beta\in\irr{I/P}\},
	\]
	and hence $\irr{I|\theta}= \irr{I|\varphi}=\{\beta\psi:\beta\in\irr{I/P}\}$ as $\irr{P|\theta}=\{\varphi\}$.
	Take $\chi\in\irr{G|\theta}$ such that $\chi\in\irr{G|\beta\psi}$.
	Then $\chi(1)\mid 2\beta(1)\psi(1)$ as $|G:I|\mid 2$.
	Recall that $K$ is the unique minimal normal subgroup of $G$,
	one has that $\ker(\chi)=1$ as $\ker(\chi)\cap K=1$.
	If $I=G$, then $\chi=\beta\psi$, and so by \cite[Theorem A]{gagolalewis1999}
	\[
		\chi(1)^2=\beta(1)^2\psi(1)^2 ~\big|~  |H||P|=|G:\ker(\chi)|
	\]
	as $H$ and $P$ being nilpotent.
	Assume that $|G:I|=2$.
	As $M\leq I$ and $|H:\cent{H}{K}|=|G:M|\mid 2$, $H$ satisfies (a4) and $I=M$.
	So, $H\cong C_{2(p+1)}$ for $p$ odd.
	Since $\beta\in\irr{I/P}$ where $I/P$ is abelian, $\beta(1)=1$,
	and so $\chi(1)\mid 2\psi(1)=2p$.
	Therefore,
	\[
		\chi(1)^2~\big|~ 4p^2 ~\big|~  |H||P|=|G:\ker(\chi)|.
	\]

	\smallskip

	\emph{Claim 3. $\irr{G|C}\subseteq\irrn{ G}$.}

	\smallskip

	Let $M=C\times P$.
	Without loss of generality, we may assume that $C>1$.
	In this case, $H$ is a Sylow $q$-subgroup of $G$ such that $q^2\mid |H|$ and $M=\fitt{G}$ has index $q$ in $G$.

	Let $\theta\in\irr{C}$ be nonprinciple and $\chi\in\irr{G|\theta}$.
	It suffices to show that $\chi\in \irrn G$.
	To see this, let $\psi\in\irr{M|\theta}$ be an irreducible constituent of $\chi_M$.
	Since $M=C\times P$, $\psi=\theta\times \varphi$ for some $\varphi\in\irr{P}$.
	As $(|C|,|P|)=1$,
	\begin{equation}\label{k}
		\ker(\psi)=\ker(\theta)\times \ker(\varphi).
	\end{equation}
	Since $|G:M|=q$, either $\chi_M=\psi$ or $\chi=\psi^G$.

	Assume that $\chi_M=\psi$.
	Then $\ker(\chi)\cap M=\ker(\psi)$.
	As $M=\fitt{G}$, by \cite[Theorem A]{gagolalewis1999}
	\[
		\chi(1)^2=\psi(1)^2~\big|~|M:\ker(\psi)|=|M:\ker(\chi)\cap M|~\big|~|G:\ker(\chi)|.
	\]

	Assume that $\chi=\psi^G$.
	Then $\ker(\chi)\leq \ker(\psi)$.
	Since $C>1$ is a $q$-group and $\theta\neq 1_C$,
	\begin{equation}\label{l}
		q~\big|~ |\z{\theta}:\ker(\theta)|~\text{and}~\theta(1)^2~\big|~ |C:\z{\theta}|.
	\end{equation}
	Also, $\varphi(1)^2\mid |P|$ as $P$ being a $p$-group.
	Therefore, by (\ref{k}) and (\ref{l})
	\[
		\chi(1)^2=q\cdot q\theta(1)^2\varphi(1)^2~\big|~q|C:\ker(\theta)||P:\ker(\varphi)|=q|M:\ker(\psi)|=|G:\ker(\psi)|~\big|~|G:\ker(\chi)|.
	\]

	\smallskip

	Consequently, $\irrs{G}=\irrs{\tilde{G}}$ forms a Galois conjugacy class.
\end{proof}



\end{document}